\documentclass[a4paper,11pt]{article}
\usepackage{amsmath,amsthm,amsfonts,amssymb,color}
\usepackage[matrix,arrow]{xy}
\usepackage{graphicx}
\usepackage[UKenglish]{babel}
\usepackage{multicol} 
\usepackage{epstopdf}

\usepackage{caption}
\usepackage{subcaption}

\usepackage{multirow} % para las tablas
\usepackage{float}
\usepackage{graphicx}
\usepackage{longtable}

\usepackage{amsmath,amssymb,amsthm}
\usepackage{enumerate}
\usepackage{amsfonts}
\usepackage{graphics,graphicx}
\usepackage{verbatim}
\usepackage{latexsym,makeidx}
\usepackage{amsmath,amscd}
\usepackage{appendix}
\usepackage[latin1]{inputenc}

\newenvironment{keywords}{
  \vspace{2mm}
  \noindent
  \keywordsname: 
  \itshape\small
}

\def\keywordsname{\textbf{Keywords}}
  \def\mathsubclassname{\textbf{2010 AMS Subject Classification}}

\newenvironment{Figure}
  {\smallskip\noindent\minipage{\linewidth}}
  {\endminipage\smallskip}

\newtheorem{teo}{Theorem}[section]

\newtheorem{lema}{Lemma}[section]

\newtheorem{obs}{Remark}[section]

\newcommand\be{\begin{equation}}
\newcommand\ee{\end{equation}}

\begin{document}

\title{Explicit solution for Stefan problem with latent heat depending on the position and a convective boundary condition at the fixed face using Kummer functions.}

\author{
Julieta Bollati$^{1}$,  Domingo A. Tarzia $^{1}$\\ \\
\small {{$^1$} Depto. Matem\'atica - CONICET, FCE, Univ. Austral, Paraguay 1950} \\  
\small {S2000FZF Rosario, Argentina.} \\
}
\date{}

\maketitle

\abstract{

An explicit solution of a similarity type is obtained for  a one-phase Stefan problem in a semi-infinite material using Kummer functions. Motivated by [D.A. Tarzia, Relationship between Neumann solutions for two phase Lam\'e-Clapeyron-Stefan problems with convective and temperature boundary conditions, Thermal Sci.(2016) DOI 10.2298/TSCI 140607003T, In press], and  [Y. Zhou, L.J. Xia, Exact solution for Stefan problem with general power-type latent heat using Kummer function, Int. J. Heat Mass Transfer, 84 (2015) 114-118], we consider a phase-change problem with a latent heat defined as a power function of the position with a non-negative real exponent and a convective boundary condition at the fixed face $x=0$.  Existence and uniqueness of the solution is proved. Relationship between this problem and the problems already solved by  Zhou and Xia with temperature and flux boundary condition is analysed. Furthermore it is studied the limit behaviour of the solution when the coefficient which characterizes the heat transfer at the fixed boundary tends to infinity. Numerical computation of the solution is done over certain examples, with a view to comparing this results with those obtained by general algorithms that solve Stefan problems.  

\vspace{0.5em}


\begin{longtable}{|p{10mm} p{3mm} p{138mm}| } 
\hline 
& &\\
{\large \textbf{Nomenclature}}&   &\\
& &\\
$c$ &  & Coefficient that characterizes the heat flux at the fixed face, $[kg/s^{(5+\alpha)/2}]$.\\
$d$ &  & Diffusivity coefficient, $[m^2/s]$.\\
$h_0$ & & Coefficient that characterizes the heat transfer in condition (\ref{4}), $[kg/(^{\circ}Cs^{5/2})]$.\\
$k$ & & Thermal conductivity, $[W/(m ^{\circ}C)]$. \\
$q,r,s,s_{\infty}$ &  & Position of the free front, $[m]$.\\
$t$& & Time, $[s]$.\\
$T$ & & Temperature, $[^{\circ}C]$.\\
$T_0$ & & Coefficient that characterizes the temperature at the fixed face, $[^{\circ}C/s^{\alpha/2}]$.\\
$T_{\infty}$ & & Coefficient that characterizes the bulk temperature, $[^{\circ}C/s^{\alpha/2}]$.\\
$x$ & & Spatial coordinate, $[m]$.\\
& &\\
$\text{Greek symbols} $ &   &\\
& &\\
$\alpha$ & & Power of the position that characterizes the latent heat per unit volume, dimensionless.\\
$\gamma$ & & Coefficient that characterizes the latent heat per unit volume, $[kg/(s^2m^{\alpha+1})]$.\\
$\lambda, \mu,\nu,\nu_{\infty}$ & & Coefficient that characterizes the free interface, dimensionless.\\
$\eta$ & & Similarity variable in expression (\ref{6}), dimensionless.\\
$\Theta, \Psi, \Psi_{\infty}$ & & Temperature, $[^{\circ}C]$.\\
& &\\
\hline
\end{longtable}

\newpage

\section{Introduction.}

The study of heat transfer problems with phase-change such as melting and freezing constitutes a broad field that has a wide engineering and industrial applications.  Stefan problems can be formulated as models that represents thermal processes in phase transitions, where these phase transitions are characterized by heat diffusion and an exchange of latent heat. Due to their  importance, they have been largely studied since the last century \cite{AlSo}-\cite{Gu},\cite{Lu},\cite{Ru} and \cite{Ta4}. In \cite{Ta2} it  was presented an extensive bibliography regarding this subject.

In the classical formulation of Stefan problems there are many assumptions  on the physical factors involved in the phase-change that are taken into account in order to simplify the description of the process. One of this hypothesis, is to consider the latent heat as a constant. Although it is a reasonable assumption, it can be dropped in order to assume a variable latent heat. For example, it can be introduced a new kind of problems where the latent heat  depends on the position. The physical bases of this particular problems can be found in the movement of a shoreline \cite{VSP}, in the ocean delta deformation \cite{LoVo} or in the cooling body of a magma \cite{Per}.

In  1970,  Primicerio \cite{Pr} gave sufficient conditions for the existence and uniqueness of solution of a one-phase Stefan problem taking a latent heat as a general function of the position.  Voller et al. \cite{VSP} in 2004 found an exact solution for a one-phase Stefan problem  considering the latent heat as a linear function of position.

On one hand, Salva and Tarzia \cite{SaTa} extended Voller's work by considering the two-phase Stefan problem with a latent heat distributed linearly on the position. On the other hand, Zhou et al. in \cite{ZWB} generalized \cite{VSP} by considering the one-phase Stefan problem with the latent heat as a power function of the position with an integer exponent. 
Recently Zhou and Xia \cite{ZhXi} worked with the latter problem  assuming a real non-negative exponent. They presented the explicit solution for two different problems 
defined according to the boundary conditions considered: temperature and flux. Explicit solutions for phase-change processes are given in \cite{Ro} and \cite{Ta3}. A recent review on the subject can be find in \cite{Ta4}.

Motivated by \cite{Ta1} and \cite{ZhXi} we are going to analyse the existence and uniqueness of solution of a one-phase Stefan problem, considering an homogeneous semi-infinite material, with a latent heat as a power function of the position and a convective boundary condition at the fixed face $x=0$. This problem can be formulated in the following way:

\noindent \textbf{Problem (P1)}: Find the temperature $\Psi(x,t)$ and the moving interface $s(t)$ such that:
\begingroup
\addtolength{\jot}{0.35 em}
\begin{eqnarray}
& &\Psi_t(x,t)=d \Psi_{xx}(x,t), \qquad 0<x<s(t), \quad t>0, \label{1}\\
& & s(0)=0,\label{2}\\
& & \Psi(s(t),t)=0, \qquad t>0, \label{3}\\
& & k\Psi_x(0,t)=h_0 t^{-1/2} \left[ \Psi(0,t)-T_{\infty}t^{\alpha/2}\right]  \qquad t>0, \label{4}\\
& & k\Psi_x(s(t),t)=-\gamma s(t)^{\alpha} \dot s(t), \qquad t>0, \label{5}
\end{eqnarray}
\endgroup
where $\Psi$ is the temperature, $s(t)$ is the moving interface, $d$ is the thermal diffusion coefficient, $k$ is the thermal conductivity,  $\gamma x^{\alpha}$ is the variable latent heat  per  unit of volume and the phase-transition temperature is zero. Condition (\ref{4}) represents the convective boundary condition at the fixed face. $T_{\infty}$ characterizes the bulk temperature at a large distance from the fixed face $x=0$ and $h_0$ represents the heat transfer at the fixed face. Moreover $\dot s(t)$ represents the velocity of the phase-change interface.
We will work under the assumption  that $\gamma>0, h_0>0 $ and  $T_{\infty}>0$ which corresponds to the melting case. In case of freezing it is sufficient to assume $h_0>0$ ,$\gamma<0$ and $T_{\infty}<0$.

The main objective of this article is to provide a detailed mathematical analysis of this heat transfer problem. In Section 2 we will use the similarity transformation technique in order to obtain an explicit solution for the problem governed by $(\ref{1})-(\ref{5})$. In Section 3 we will present  a relationship between the problem (P1) and the two related problems with temperature and heat flux boundary conditions on the fixed face $x=0$ studied in \cite{ZhXi} . Section 4 deals with the limit behaviour of the solution of (P1) when the coefficient that characterizes the heat transfer at the fixed face tends to infinity.  Finally some computational examples will be shown in Section 5.

\vspace{0.5cm}
\section{Explicit solution.}
\vspace{0.5cm}

\subsection{General case when $\alpha$ is a  non-negative real exponent.}

The following lemma have already been developed by Zhou-Xia in \cite{ZhXi}. It is going to be useful in order to find solutions for the differential heat equation (\ref{1}).
\vspace{0.2cm}

\begin{lema}  \cite{ZhXi} \label{Lemma2.1}

\begin{enumerate}[a.]
\item Let 
\be \label{6}
\Psi(x,t)=t^{\alpha/2}f(\eta), \text{ with \quad} \eta=\dfrac{x}{2\sqrt{dt}}
\ee

then $\Psi=\Psi(x,t)$ is a solution of the heat equation $\Psi_t(x,t)=d\Psi_{xx}(x,t)$, with $d>0$ if and only if $f=f(\eta)$ satisfies the following ordinary differential equation:

\be \label{diffeq1}
\frac{d^2f}{d\eta^2}(\eta)+2\eta \frac{df}{d\eta}(\eta)-2\alpha f(\eta)=0.
\ee

\item An equivalent formulation for equation (\ref{diffeq1}), introducing the new variable $z=-\eta^2$, is given by:

\be \label{diffeq2}
z\frac{d^2f}{dz^2}(z)+\left(\frac{1}{2}-z \right)\frac{df}{dz}(z)+\frac{\alpha}{2}f(z)=0.
\ee

\item  The general solution of the ordinary differential equation (\ref{diffeq2}), called Kummer's equation, is given by:

\be\label{genSol1}
f(z)=\widehat{c_{11}}M \left(-\dfrac{\alpha}{2},\dfrac{1}{2},z\right)+\widehat{c_{21}}U\left(-\dfrac{\alpha}{2},\dfrac{1}{2},z\right) .
\ee
where $\widehat{c_{11}}$ and $\widehat{c_{21}}$ are arbitrary real constants and $M(a,b,z)$ and $U(a,b,z)$ are the Kummer functions defined by:
\begin{align}
& M(a,b,z)=\sum\limits_{s=0}^{\infty}\frac{(a)_s}{(b)_s s!}z^s, \text{ where b cannot be a non-positive integer,} \label{M} \\
& U(a,b,z)=\frac{\Gamma(1-b)}{\Gamma(a-b+1)}M(a,b,z)+\frac{\Gamma(b-1)}{\Gamma(a)} z^{1-b}M(a-b+1,2-b,z) \label{U}.
\end{align}
where $(a)_s$ is the pochhammer symbol  defined by:
\be 
 (a)_s=a(a+1)(a+2)\dots (a+s-1), \quad \quad (a)_0=1 
\ee

\end{enumerate}

\end{lema}

\begin{obs}
All the properties of Kummer's functions to be used in the following arguments can be found in \cite{OLBC}.

\end{obs}

\begin{obs} \label{Remark1}
Taking into account  definition (\ref{U}) we can rewrite the general solution of (\ref{diffeq2}) as:

\be\label{genSol2}
f(z)=c_{11}M \left(-\dfrac{\alpha}{2},\dfrac{1}{2},z\right)+c_{21}z^{1/2} M\left(-\dfrac{\alpha}{2}+\dfrac{1}{2},\dfrac{3}{2},z\right),
\ee
where $c_{11}$ and $c_{21}$ are real constants.
\end{obs}

Our main outcome is given by the following theorem which ensures the existence and uniqueness of solution of the problem (P1) providing in addition, the corresponding explicit solution.

\begin{teo} \label{Teo2.1}
There exists a unique solution  of a similarity type for the one-phase Stefan problem (\ref{1})-(\ref{5}) and it is  given by:
\begingroup
\addtolength{\jot}{0.35em}
\begin{flalign}
&\Psi(x,t)=  t^{\alpha/2}\left[ c_{11} M\left(-\frac{\alpha}{2}, \frac{1}{2},-\eta^2\right)+c_{21} \eta M\left(-\frac{\alpha}{2}+\frac{1}{2},\frac{3}{2},-\eta^2\right)\right] \label{14}\\
&s(t) =2\nu \sqrt{d t} \label{15}
\end{flalign}
\endgroup
where  $\eta=\dfrac{x}{2\sqrt{dt}}$ and the constants $c_{11}$ and $c_{21}$ are given by:

\begingroup
\addtolength{\jot}{0.35em}
\begin{flalign}
& c_{11}=\dfrac{-\nu M\left(-\dfrac{\alpha}{2}+\dfrac{1}{2},\dfrac{3}{2},-\nu^2\right)}{M\left(-\dfrac{\alpha}{2},\dfrac{1}{2},-\nu^2\right)}c_{21}, \label{16} \\
& c_{21}=\dfrac{-2h_0\sqrt{d} T_{\infty} M\left( -\dfrac{\alpha}{2},\dfrac{1}{2},-\nu^2\right)}{\left[k M\left( -\dfrac{\alpha}{2},\dfrac{1}{2},-\nu^2\right)+2\sqrt{d}h_0 \nu M\left( -\dfrac{\alpha}{2}+\dfrac{1}{2},\dfrac{3}{2},-\nu^2\right) \right]} \label{17}.
\end{flalign}
\endgroup
and the dimensionless coefficient $\nu$ is obtained as the unique positive solution of the following equation:

\begin{eqnarray} \label{18}
\dfrac{h_0 T_{\infty}}{\gamma 2^{\alpha} d^{(\alpha+1)/2}}f_1(x)=x^{\alpha+1},\qquad \qquad x>0.
\end{eqnarray}
in which:

\begingroup
\addtolength{\jot}{0.35em}
\begin{align}
& f_1(x)=\dfrac{1}{\left[ M\left(\dfrac{\alpha}{2}+\dfrac{1}{2},\dfrac{1}{2},x^2\right)+2\dfrac{\sqrt{d}h_0}{k}x M\left(\dfrac{\alpha}{2}+1,\dfrac{3}{2},x^2\right)\right]}. \label{19}
\end{align}
\endgroup

\end{teo}

\begin{proof}
The general solution of equation (\ref{1}), based on Kummer functions is given by the Lemma \ref{Lemma2.1} . According to Remark \ref{Remark1} we can write:
\begingroup
\addtolength{\jot}{0.35em}
\begin{align}
& \Psi(x,t)=t^{\alpha/2}\left[c_{11} M\left(-\dfrac{\alpha}{2},\dfrac{1}{2},-\eta^2\right)+c_{21} \eta M\left(-\dfrac{\alpha}{2}+\dfrac{1}{2},\dfrac{3}{2},-\eta^2\right)  \right], \label{20}
\end{align}
\endgroup
 where $\eta=\dfrac{x}{2\sqrt{dt}}$ and where $c_{11}$ and $c_{21}$ are coefficients that must be determined in order to ensure that $\Psi$ satisfies conditions (\ref{3})-(\ref{5}).

Furthermore, condition (\ref{3}) together with (\ref{20}) implies that the free boundary should take the following  form:

\be \label{21}
s(t)=2\nu \sqrt{d t}. 
\ee
where $\nu$ is a constant that have to be determined.

From  equations (\ref{3}), (\ref{20}) and (\ref{21}) we obtain that:

\begin{align}
 &\Psi(s(t),t)= t^{\alpha/2}\left[ c_{11}M\left(-\dfrac{\alpha}{2},\dfrac{1}{2},-\nu^2 \right)+c_{21}\nu M\left(-\dfrac{\alpha}{2}+\dfrac{1}{2},\dfrac{3}{2},-\nu^2 \right)\right]=0,
\end{align}
and isolating $c_{11}$ we arrive to (\ref{16}).

On the other hand, we know that the derivative of the Kummer functions (see \cite{OLBC}) are :

\begingroup
\addtolength{\jot}{0.35em}
\begin{flalign}
\dfrac{d}{dz}M(a,b,z) &=\dfrac{a}{b}M(a+1,b+1,z) ,\label{24} \\
 \dfrac{d}{dz}z^{b-1}M(a,b,z)&=(b-1)z^{b-2}M(a,b-1,z) \label{25},
\end{flalign}
\endgroup
and  therefore by using (\ref{24}) and (\ref{25}) we get that:

\be
{\Psi}_x(x,t)=\dfrac{t^{(\alpha-1)/2}}{\sqrt{d}} \left[c_{11} \alpha \eta M\left( -\dfrac{\alpha}{2}+1,\dfrac{3}{2},-\eta^2\right)+\dfrac{c_{21}}{2}M\left(-\dfrac{\alpha}{2}+\dfrac{1}{2},\dfrac{1}{2},-\eta^2 \right) \right] ,\label{26} 
\ee
 and in consequence, condition (\ref{4}) is satisfied if and only if:

\be \label{27}
k \dfrac{t^{(\alpha-1)/2}}{2\sqrt{d}} c_{21}=h_0 t^{-1/2}\left[ t^{\alpha/2}c_{11}-T_{\infty}t^{\alpha/2}\right],
\ee
that is:
\be \label{27Bis}
k \dfrac{c_{21}}{2\sqrt{d}} =h_0 \left[ c_{11}-T_{\infty}\right].
\ee
Replacing $c_{11}$ given by (\ref{16}) into (\ref{27Bis}) we find (\ref{17}).

Until now we have obtained $c_{11}$ and $c_{21}$ as functions of $\nu$, arriving to the expressions (\ref{16})-(\ref{17}).  By combining equations (\ref{16}), (\ref{17}), (\ref{21}) and (\ref{26})  and using the following identities \cite{ZhXi}:

\begin{align}
& M(a,b,z)=e^zM(b-a,b,-z),   \\
& e^{-\nu^2}=  -2\alpha \nu^2 M\left( -\dfrac{\alpha}{2}+\dfrac{1}{2},\dfrac{3}{2},-\nu^2\right)M\left(-\dfrac{\alpha}{2}+1,\dfrac{3}{2},-\nu^2 \right) +M\left( -\dfrac{\alpha}{2}+\dfrac{1}{2},\dfrac{1}{2},-\nu^2\right)M\left(-\dfrac{\alpha}{2},\dfrac{1}{2},-\nu^2 \right) \label{29} ,
\end{align}
we obtain that the Stefan condition (\ref{5}) holds if and only if $\nu $ satisfies the equation:

\be
\dfrac{kh_0T_{\infty}}{\left[ kM\left( \dfrac{\alpha}{2}+\dfrac{1}{2},\dfrac{1}{2},x^2\right)+2\sqrt{d}h_0x M\left( \dfrac{\alpha}{2}+1,\dfrac{3}{2},x^2\right)\right]}=\gamma x^{\alpha+1}2^{\alpha}d^{(\alpha+1)/2}, \qquad \qquad x>0. \label{30}
\ee

It means that $\Psi$ and $s$ defined in (\ref{14}) and (\ref{15}) constitute a solution of problem (\ref{1})-(\ref{5}), with $c_{11}$ and $c_{22}$ given by equations (\ref{16}) and (\ref{17}) if and only if $\nu$, the unknown coefficient, verifies the equation (\ref{30}). Thus we have deduced an equality that must be satisfied  by the positive coefficient $\nu$ and that can be written as:

\begin{eqnarray} 
\dfrac{h_0 T_{\infty}}{\gamma 2^{\alpha} d^{(\alpha+1)/2}}f_1(\nu)=\nu^{\alpha+1}, \label{31}
\end{eqnarray}
where the real function $f_1$ is defined by (\ref{19}).

The proof will be completed by showing the existence and uniqueness of solution to equation (\ref{18}) i.e equation (\ref{31}),  analysing the monotonicity of the left and the right hand side of this equality.

By using properties (\ref{24}) and (\ref{25}) of Kummer functions we can observe that:

\begingroup
\addtolength{\jot}{0.35em}
\begin{align}
 f_1'(x) =& -\left[ 2(\alpha+1)x M\left(\dfrac{\alpha}{2}+\dfrac{3}{2},\dfrac{3}{2},x^2 \right)+2\dfrac{\sqrt{d}}{k}h_0 M\left( \dfrac{\alpha}{2}+1,\dfrac{1}{2},x^2\right) \right] f_1^2(x)  <0, \qquad \qquad \forall x>0.
\end{align}
\endgroup

Therefore we can assure that $f_1(x)$ is a decreasing function of $x$. Consequently, the left hand side of (\ref{18}), is also a decreasing function of $x$ that goes from $\dfrac{h_0T_{\infty}}{\gamma2^{\alpha}d^{(\alpha+1)/2}}>0$ to 0 when $x$ increases from 0 to $+\infty$. Meanwhile the right hand side of (\ref{18}) is an increasing function of $x$ that increases from $0$ to $+\infty$, when $x$ goes from $0$ to $+\infty$.

The above assertions allow us to conclude that there always exists a unique positive  solution $\nu$ of (\ref{18}) regardless of the data. Then we obtain that the problem (\ref{1})-(\ref{5}) always
 has a unique solution given by (\ref{14})-(\ref{19}).

\vspace{1em}

\end{proof}
\vspace{0.5cm}
\subsection{Special case when $\alpha$ is an integer.}
 In the special case that $\alpha$ is a positive integer, denoted by $n$, the Kummmer functions are related with  the iterated integral of the complementary error function and with the gamma function as follows (see \cite{ZhXi}, \cite{OLBC}):
 
 \begingroup
\addtolength{\jot}{0.35em}
\begin{align}
& M\left( -\dfrac{n}{2},\dfrac{1}{2},-z^2\right)=2^n \Gamma\left( \dfrac{n}{2}+1\right)E_n(z), \label{34} \\
& zM\left(-\dfrac{n}{2}+\dfrac{1}{2},\dfrac{3}{2},-z^2 \right)=2^{n-1}\Gamma\left( \dfrac{n}{2}+\dfrac{1}{2}\right)F_n(z), \label{35}.
\end{align}
\endgroup

\noindent where :

\begingroup
\addtolength{\jot}{0.35em}
\begin{align}
&E_n(z)=\dfrac{\left[ i^n erfc(z)+i^n erfc(-z)\right]}{2},\label{36}\\
&F_n(z)=\dfrac{\left[ i^n erfc(-z)-i^n erfc(z)\right]}{2}. \label{37}
\end{align}
\endgroup

\noindent Such properties allow us to transform the solution of the problem (\ref{1})-(\ref{5}), given by Theorem \ref{Teo2.1}, in case that $\alpha=n\in \mathbb{N}$ into:

\begingroup
\addtolength{\jot}{0.35em}
\begin{align}
& \Psi(x,t)=\dfrac{-t^{n/2}2^n h_0 T_{\infty} \sqrt{d} \Gamma\left( \dfrac{n}{2}+\dfrac{1}{2}\right) \Gamma\left( \dfrac{n}{2}+1\right) \left[ F_n(\eta)E_n(\nu)-F_n(\nu)E_n(\eta) \right]}{k\Gamma\left(\dfrac{n}{2}+1 \right)E_n(\nu)+\sqrt{d}h_0\Gamma\left( \dfrac{n}{2}+\dfrac{1}{2}\right)F_n(\nu)}, \label{38} \\
& s(t)=2\nu \sqrt{dt}, \label{39}
\end{align}
 where $\eta= \dfrac{x}{2\sqrt{dt}}$ and $\nu$ is the unique positive solution of the following equation:

\be
\dfrac{h_0T_{\infty} }{\gamma d^{(n+1)/2} 2^{2n}\left[ \Gamma\left( \dfrac{n}{2}+1\right)E_n(x)+\sqrt{d}\dfrac{h_0}{k}\Gamma\left( \dfrac{n}{2}+\dfrac{1}{2}\right)F_n(x)\right]}=x^{n+1}e^{x^2}. \label{40}
\ee

 \begin{obs} Taking into account that $E_0(x)=1$ and $F_0(x)=erf(x)$, in the case $\alpha=0$, functions (\ref{38})-(\ref{39}) and equation (\ref{40}) reduce to:

 \begingroup
\addtolength{\jot}{0.35em}
\begin{align}
& \Psi(x,t)= \dfrac{-h_0T_{\infty} \sqrt{d} \sqrt{\pi} \left[ erf\left(-\dfrac{x}{2\sqrt{dt}}\right) -erf(\nu)\right]}{k\left[ 1+\dfrac{\sqrt{d \pi}h_0}{k} erf(\nu)\right]}, \\
& s(t)=2\nu \sqrt{dt} ,
\end{align}
\endgroup
where $\nu$ is the unique positive solution of:
 
 \be
\dfrac{h_0T_{\infty} }{\gamma \sqrt{d} \left[ 1+\dfrac{\sqrt{d \pi}h_0 }{k} erf(x)\right]}=x e^{x^2}, \qquad \qquad x>0. 
\ee 
 It can be noted that this solution is in accordance with the solution given by Tarzia \cite{Ta1} in case that initial temperature $T_i=0$  (reducing the two-phase Stefan problem into a one-phase Stefan problem).
 
  \end{obs}

 \vspace{0.5cm}
 
 \section{Equivalence between problems with temperature, flux and convective boundary conditions at the fixed face $x=0$.}

We denote by (P1) the problem governed by (\ref{1})-(\ref{5}). If we change the convective  condition (\ref{4}) by a temperature boundary condition we obtain a problem that will be denote by (P2) whose explicit solution was presented in \cite{ZhXi}.  Similarly we can define the  problem (P3) changing condition (\ref{4}) by a flux boundary condition, whose exact solution was also presented by Zhou and Xia in \cite{ZhXi}. It means that we have defined:

\vspace{1em}

\noindent \textbf{Problem (P2)}: Find the temperature $\Theta(x,t)$ and the moving interface $r(t)$ that satisfies:
\begingroup
\addtolength{\jot}{0.35em}
\begin{align}
& \Theta_t(x,t)=d \Theta_{xx}(x,t), \qquad 0<x<r(t), \quad t>0, \label{p2-1}\\
&  r(0)=0,\label{p2-2}\\
&  \Theta(r(t),t)=0, \qquad t>0, \label{p2-3}\\
&  \Theta(0,t)=T_0t^{\alpha/2}  \qquad t>0, \label{p2-4}\\
&  k\Theta_x(r(t),t)=-\gamma r(t)^{\alpha} \dot r(t), \qquad t>0, \label{p2-5}
\end{align}
\endgroup

\noindent where the solution according to \cite{ZhXi} is given by :

\begingroup
\addtolength{\jot}{0.35em}
\begin{align}
&\Theta(x,t)=  t^{\alpha/2}\left[ c_{12} M\left(-\frac{\alpha}{2}, \frac{1}{2},-\eta^2\right)+c_{22} \eta M\left(-\frac{\alpha}{2}+\frac{1}{2},\frac{3}{2},-\eta^2\right)\right], \label{49}\\
& r(t)=2\mu  \sqrt{dt}, \label{50}
\end{align}
\endgroup
 where $\eta= \dfrac{x}{2\sqrt{dt}}$ and the constants $c_{12}$ and $c_{22}$ are given by: 

\begingroup
\addtolength{\jot}{0.35em}
\begin{align}
& c_{12}=T_0, \qquad \qquad \qquad  c_{22}=\dfrac{-T_0M\left( -\dfrac{\alpha}{2},\dfrac{1}{2},-\mu^2\right)}{\mu M\left(-\dfrac{\alpha}{2}+\dfrac{1}{2},\dfrac{3}{2},-\mu^2 \right)} \label{51}
\end{align}
\endgroup
and the parameter $\mu$ is the unique positive solution of the following equation:

\be \label{52}
\dfrac{kT_0}{2^{\alpha+1} d^{\alpha/2+1} \gamma } f_2(x)=x^{\alpha +1},   \qquad x>0
\ee
with:
\be \label{52bis}
f_2(x)=\dfrac{1}{x M\left(\dfrac{\alpha}{2}+1,\dfrac{3}{2},x^2 \right)} .
\ee

\vspace{1em}

\noindent \textbf{Problem (P3)}: Find the temperature $T(x,t)$ and the moving interface $q(t)$ such as:
\begingroup
\addtolength{\jot}{0.35em}
\begin{align}
& T_t(x,t)=d T_{xx}(x,t), \qquad 0<x<q(t), \quad t>0, \label{p2-1}\\
& q(0)=0,\label{p2-2}\\
&T(q(t),t)=0, \qquad t>0, \label{p2-3}\\
& kT_x(0,t)=  -ct^{(\alpha-1)/2}\qquad t>0, \label{p2-4}\\
& kT_x(q(t),t)=-\gamma q(t)^{\alpha} \dot q(t), \qquad t>0, \label{p2-5}
\end{align}
\endgroup

\noindent where the solution according to \cite{ZhXi} is given by:

\begingroup
\addtolength{\jot}{0.35em}
\begin{align}
&T(x,t)=  t^{\alpha/2}\left[ c_{13} M\left(-\frac{\alpha}{2}, \frac{1}{2},-\eta^2\right)+c_{23} \eta M\left(-\frac{\alpha}{2}+\frac{1}{2},\frac{3}{2},-\eta^2\right)\right]  \label{58}\\
& q(t)=2\lambda  \sqrt{dt} \label{59}
\end{align}
\endgroup
 where $\eta= \dfrac{x}{2\sqrt{dt}}$ and the constants $c_{13}$ and $c_{23}$ are given by:

\begingroup
\addtolength{\jot}{0.35em}
\begin{align}
& c_{13}=\dfrac{-\lambda M\left( -\dfrac{\alpha}{2}+\dfrac{1}{2},\dfrac{3}{2},-\lambda^2 \right)}{M\left(-\dfrac{\alpha}{2},\dfrac{1}{2},-\lambda^2 \right)} c_{23}, \qquad \qquad \qquad  c_{23}=\dfrac{-2c\sqrt{d}}{k}. \label{60}
\end{align}
\endgroup
and $\lambda$ is the unique positive solution of the following equation:

\be \label{61}
\dfrac{c}{\gamma 2^{\alpha} d^{(\alpha+1)/2}} f_3(x)=x^{\alpha+1}, \quad x>0
\ee
where

\be \label{61bis}
f_3(x)=\dfrac{1}{ M\left(\dfrac{\alpha}{2}+\dfrac{1}{2},\dfrac{1}{2},x^2 \right)}.
\ee

Once we have defined our three problems, we are going to prove the equivalence between them. We refer to equivalence in the sense that if the data of both problems satisfy certain relationship then they have the same solution.

\begin{teo} \label{P1P2}The free boundary problems (P1) and (P2) are equivalents. Moreover we have:

\begin{enumerate}[a) ]
\item the relationship between the datum  $T_0$ of problem (P2) with the data $T_{\infty}$ and $h_0$ of the problem (P1) is given by:

\be \label{62}
T_0=\dfrac{2\sqrt{d} h_0T_{\infty} \nu M\left( -\dfrac{\alpha}{2}+\dfrac{1}{2},\dfrac{3}{2},-\nu^2 \right) }{k M\left(-\dfrac{\alpha}{2},\dfrac{1}{2},-\nu^2 \right)+2\sqrt{d}h_0 \nu M\left(-\dfrac{\alpha}{2}+\dfrac{1}{2},\dfrac{3}{2},-\nu^2 \right)} .
\ee
where $\nu$ is the parameter that characterizes the moving interface in problem (P1) and it is given as the unique solution of the equation (\ref{18}).

\item the relationship between the  data $h_0$ and $T_{\infty}$ of problem (P1) with the datum $T_0$ of the problem (P2) is given by $T_{\infty}>T_0$ and :

\be \label{63}
h_0= \dfrac{-kT_0 M\left(-\dfrac{\alpha}{2},\dfrac{1}{2},-\mu^2 \right)}{2\sqrt{d}(T_0-T_\infty)\mu M\left(-\dfrac{\alpha}{2}+\dfrac{1}{2},\dfrac{3}{2},-\mu^2 \right)} .
\ee
where $\mu$ is the parameter that characterizes the moving interface in problem (P2) and it is given as the unique solution of the equation (\ref{52}).

\end{enumerate}
\end{teo}

\begin{proof}
$ $
\begin{enumerate} [a) ]
\item First, we solve the free boundary problem (P1) and we obtain $\Psi(x,t)$, and $s(t)$ through equations (\ref{14})-(\ref{19}). If we compute the temperature of this problem at the fixed face $x=0$ we get:

\begin{align}
\Psi(0,t)=t^{\alpha/2}c_{11}& = t^{\alpha/2} \dfrac{2\sqrt{d} h_0T_{\infty} \nu M\left( -\dfrac{\alpha}{2}+\dfrac{1}{2},\dfrac{3}{2},-\nu^2 \right) }{k M\left(-\dfrac{\alpha}{2},\dfrac{1}{2},-\nu^2 \right)+2\sqrt{d}h_0 \nu M\left(-\dfrac{\alpha}{2}+\dfrac{1}{2},\dfrac{3}{2},-\nu^2 \right)},
\end{align}
so it leads us to define   $T_0$ as $\dfrac{\Psi(0,t)}{t^{\alpha/2}}$  arriving to (\ref{62}). Observe that $\nu$ is the parameter which defines $s(t)$ (the moving interface of problem (P1)) and it is the unique solution of (\ref{18}).

Considering the problem (P2) with this particular $T_0$, defined by (\ref{62}), we obtain that the temperature $\Theta(x,t)$ and the moving interface $r(t)$ are given by (\ref{49})-(\ref{52bis}). From this equations we have that the parameter $\mu$ which characterizes $r(t)$ is the unique solution of :
\be \label{65}
\dfrac{k}{2^{\alpha+1} d^{\alpha/2+1} \gamma } \dfrac{2\sqrt{d} h_0T_{\infty} \nu M\left( -\dfrac{\alpha}{2}+\dfrac{1}{2},\dfrac{3}{2},-\nu^2 \right) }{\left[k M\left(-\dfrac{\alpha}{2},\dfrac{1}{2},-\nu^2 \right)+2\sqrt{d}h_0 \nu M\left(-\dfrac{\alpha}{2}+\dfrac{1}{2},\dfrac{3}{2},-\nu^2 \right)\right]}  f_2(x)=x^{\alpha +1}, \qquad x>0.  
\ee 
If we replace $x$ by $\nu$ in equation (\ref{65}) we obtain equation (\ref{18}) whose unique solution is $\nu$. So we can conclude that $\nu$ is a solution of (\ref{65})
. Therefore we get that $\mu=\nu$, and $r(t)=s(t)$. Working algebraically we obtain that the temperature of both problems are equal, i.e $\Theta(x,t)=\Psi(x,t)$. In other words,  the problem (P1) has the same solution of  problem (P2) when  $T_0$ is defined in function of the data of (P1) as (\ref{62}).

\item Conversely, we consider the problem (P2), and we solve it using equations (\ref{49})-(\ref{52bis}), we obtain $\Theta(x,t)$ and $r(t)$. If we compute $\Theta(0,t)$ and $\Theta_x(0,t)$, the  coefficient $h_0$ can be defined in order that convective condition (\ref{4}) is satisfied. That is to say:

\begin{align}
h_0& = \dfrac{k\Theta_x(0,t)}{t^{-1/2}\left[ \Theta(0,t)-T_{\infty}t^{\alpha/2} \right]} \\
& =\dfrac{-kt^{(\alpha-1)/2} T_0 M\left(-\dfrac{\alpha}{2},\dfrac{1}{2},-\mu^2 \right)}{2\sqrt{d} \mu M\left( -\dfrac{\alpha}{2}+\dfrac{1}{2},\dfrac{3}{2},-\mu^2\right)t^{-1/2}\left[t^{\alpha/2}T_0-t^{\alpha/2}T_{\infty}  \right]} 
\end{align}
arriving to definition (\ref{63}), where $\mu$ is the parameter that characterizes the moving interface $r(t)$, and it is the unique solution of (\ref{52}).

Imposing a  $T_{\infty}>T_0$,  it turns out that $h_0$ defined by (\ref{63}) is positive, and hence  we can  solve the problem (P2) with this $h_0$. By equations (\ref{14})-(\ref{19}) we obtain the temperature $\Psi(x,t)$ and the moving interface $s(t)=2\nu\sqrt{dt}$.  From (\ref{18}) and taking into account the form of $h_0$ we get that $\nu$  is the unique solution of:

\be \label{69}
\dfrac{-kT_0 M\left(-\dfrac{\alpha}{2},\dfrac{1}{2},-\mu^2 \right)}{2\sqrt{d}(T_0-T_\infty)\mu M\left(-\dfrac{\alpha}{2}+\dfrac{1}{2},\dfrac{3}{2},-\mu^2 \right)} \dfrac{T_{\infty}}{\gamma 2^{\alpha} d^{(\alpha+1)/2}} f_1(x)=x^{\alpha+1}, \qquad x>0.
\ee

If we replace $x$ by $\mu$ in  equation (\ref{69}) we obtain equation (\ref{52}). As $\mu$ is the unique solution of (\ref{52}),  we obtain that $\mu$ is a solution of (\ref{69}). By uniqueness of solution of equation (\ref{69}) we get that $\nu=\mu$. In consequence, if follows that $s(t)=r(t)$ and $\Psi(x,t)=\Theta(x,t)$. So we can claim to have for  the problem (P2)  the same solution as for the  problem (P1) considering  $h_0$ defined by (\ref{63})  in function of the data of (P2).

\noindent Therefore, we can conclude that problems (P1) and (P2) are equivalents.
\end{enumerate}

\end{proof}

It remains to prove that (P1) and (P3) are also equivalents in the same way we have done for Theorem  \ref{P1P2}.

\begin{teo} The free boundary problems (P1) and (P3) are equivalents. Moreover we have:

\begin{enumerate}[a) ]
\item the relationship between the datum  $c$ of problem (P2) with the data $T_{\infty}$ and $h_0$ of the problem (P1) is given by:

\be \label{70}
c=\dfrac{  h_0  T_{\infty}M\left(-\dfrac{\alpha}{2},\dfrac{1}{2},-\nu^2 \right)  }{ \left[ M\left( -\dfrac{\alpha}{2},\dfrac{1}{2},-\nu^2\right)+\dfrac{2\sqrt{d}h_0}{k} \nu M\left( -\dfrac{\alpha}{2}+\dfrac{1}{2},\dfrac{3}{2},-\nu^2\right)\right]}
\ee
where $\nu$ is the parameter that characterizes the moving interface in problem (P1).

\item the relationship between the  data $h_0$ and $T_{\infty}$ of problem (P1) with the datum $c$ of the problem (P3) is given by :

\begin{align}
&T_{\infty}>  \dfrac{2c\sqrt{d}}{k}\dfrac{\lambda M\left(-\dfrac{\alpha}{2}+\dfrac{1}{2},\dfrac{3}{2},-\lambda^2 \right)}{M\left( -\dfrac{\alpha}{2},\dfrac{1}{2},-\lambda^2\right)}\label{71}\\
&h_0= \dfrac{-cM\left(-\dfrac{\alpha}{2},\dfrac{1}{2},-\lambda^2 \right)}{\dfrac{2c\sqrt{d}}{k} \lambda M\left( -\dfrac{\alpha}{2}+\dfrac{1}{2},\dfrac{3}{2},-\lambda^2\right)-T_{\infty} M\left(-\dfrac{\alpha}{2},\dfrac{1}{2},-\lambda^2 \right)} \label{72}.
\end{align}
where $\lambda$ is the parameter that characterizes the moving interface in problem (P3).

\end{enumerate}
\end{teo}

\begin{proof}
$ $
\begin{enumerate} [a) ]
\item First, we solve the free boundary problem (P1) and we obtain $\Psi(x,t)$, and $s(t)$ through equations (\ref{14})-(\ref{19}). If we compute the flux $\Psi$ at the fixed face $x=0$ we get:

\begin{align} 
\Psi_x(0,t)=\dfrac{-t^{(\alpha-1)/2} h_0  T_{\infty} M\left(-\dfrac{\alpha}{2},\dfrac{1}{2},-\nu^2 \right)}{k \left[ M\left(-\dfrac{\alpha}{2},\dfrac{1}{2},-\nu^2 \right)+\dfrac{ 2\sqrt{d} h_0}{k} M\left( -\dfrac{\alpha}{2}+\dfrac{1}{2},\dfrac{3}{2},-\nu^2\right)\right]} \label{73}
\end{align}
so it leads us to define  $c=-\dfrac{k \Psi_x(0,t)}{t^{(\alpha-1)/2}}$  as in (\ref{70}). Observe that $\nu$ is the parameter which defines $s(t)$ (the moving interface of problem (P1)) and it is the unique solution of (\ref{18}).

If we consider the problem (P3) with this particular $c$ defined by (\ref{70}), we obtain that the solution, it means the temperature $T(x,t)$ and the moving interface $q(t)$ are given by (\ref{58})-(\ref{61bis}). From this equations we have that the parameter $\lambda$ which characterizes $q(t)$ is the unique solution of :
\be \label{74}
\dfrac{  h_0  T_{\infty}M\left(-\dfrac{\alpha}{2},\dfrac{1}{2},-\nu^2 \right)  }{ \left[ M\left( -\dfrac{\alpha}{2},\dfrac{1}{2},-\nu^2\right)+\dfrac{2\sqrt{d}h_0}{k} \nu M\left( -\dfrac{\alpha}{2}+\dfrac{1}{2},\dfrac{3}{2},-\nu^2\right)\right]}\dfrac{1}{\gamma 2^{\alpha}d^{(\alpha+1)/2} } f_3(x)=x^{\alpha+1}, \qquad x>0.
\ee 
If we replace $x$ by $\nu$ we can reduce equation (\ref{74}) into (\ref{18}), and  as $\nu$ is the unique solution of (\ref{18}), we deduce that  $\nu$ is a solution of (\ref{74}).
. Therefore we get that $\lambda=\nu$, and $q(t)=s(t)$. Working algebraically we obtain that the temperature of both problems are equal, i.e $T(x,t)=\Psi(x,t)$. In other words,  the problem (P1) has the same solution of  problem (P3) considering a $c$ defined by (\ref{70}).

\item Conversely, if we take the problem (P3), and we solve it using equations (\ref{58})-(\ref{61bis}), we obtain $T(x,t)$ and $q(t)$.   For convective condition (\ref{4}) to happen, we compute $T(0,t)$ and $T_x(0,t)$ and define $h_0$ as:

\begin{align}
h_0& = \dfrac{kT_x(0,t)}{t^{-1/2}\left[ T(0,t)-T_{\infty}t^{\alpha/2} \right]} \\
&= \dfrac{k}{2\sqrt{d}} \dfrac{(-2c\sqrt{d})}{k}\dfrac{M\left(-\dfrac{\alpha}{2},\dfrac{1}{2},-\lambda^2 \right)}{\left[ \dfrac{2c\sqrt{d}}{k} \lambda M\left(-\dfrac{\alpha}{2}+\dfrac{1}{2},\dfrac{3}{2},-\lambda^2 \right)-T_{\infty} M\left( -\dfrac{\alpha}{2},\dfrac{1}{2},-\lambda^2\right)\right]}
\end{align}
arriving to an $h_0$ given by (\ref{72}). Observe that $\lambda$ is the parameter that characterizes the moving interface $q(t)$, which is the unique solution of (\ref{61}).

Prescribing a $T_{\infty}$ as in (\ref{71}), we are able to ensure that $h_0>0$. Hence we can pose the problem (P3) with  $h_0$ defined by (\ref{72}). By equations (\ref{14})-(\ref{19}) we obtain the temperature $\Psi(x,t)$ and the moving interface $s(t)=2\nu\sqrt{dt}$.  From (\ref{18}) and taking into account the form of $h_0$ we get that $\nu$  is the unique solution of:

\be \label{76}
 \dfrac{-cM\left(-\dfrac{\alpha}{2},\dfrac{1}{2},-\lambda^2 \right)}{\left[ \dfrac{2c\sqrt{d}}{k} \lambda M\left( -\dfrac{\alpha}{2}+\dfrac{1}{2},\dfrac{3}{2},-\lambda^2\right)-T_{\infty} M\left(-\dfrac{\alpha}{2},\dfrac{1}{2},-\lambda^2 \right)\right]} \dfrac{T_{\infty}}{\gamma 2^{\alpha} d^{(\alpha+1)/2} } f_1(x)=x^{\alpha+1}.
 \ee
If we replace $x$ by $\lambda$, equation (\ref{76}) reduces to equation (\ref{61}). As $\lambda$ is the unique solution of (\ref{61}),  we obtain that $\lambda$ is a solution of (\ref{76}). By uniqueness of solution of equation (\ref{76}) we get that $\nu=\lambda$. In consequence, if follows that $s(t)=q(t)$ and $\Psi(x,t)=T(x,t)$. It yields that the problem (P3) has the same solution of the problem (P1) when $h_0$ and $T_{\infty}$ are defined from the data of (P3) by equations (\ref{71})-(\ref{72}).

Thus we can conclude that (P1) and (P3) are equivalents.
\end{enumerate}

\end{proof}

\vspace{0.5cm}

\section{Limit behaviour}

In this section we are going to analyse the behaviour of the problem (P1) when the coefficient $h_0$ that characterizes the heat transfer at the fixed face $x=0$  tends to infinity. Due to the fact that the solution of this problem, i.e the temperature and the free boundary depends on $h_0$, we will rename them. Thus, we will consider $\Psi_{h_0}(x,t):=\Psi(x,t)$ and $s_{h_0}(t):=s(t)$ defined by equations (\ref{14})-(\ref{15}), where $c_{11}=c_{11}(h_0)$, $c_{21}=c_{21}(h_0)$ and $\nu=\nu_{h_0}$ is the unique solution of the following equation:

\begin{eqnarray} \label{NuInf}
\dfrac{h_0 T_{\infty}}{\gamma 2^{\alpha} d^{(\alpha+1)/2}}f_1(x, h_0)=x^{\alpha+1},\qquad \qquad x>0.
\end{eqnarray}

\noindent in which:

\begingroup
\addtolength{\jot}{0.35em}
\begin{align}
& f_1(x,h_0)=\dfrac{1}{\left[ M\left(\dfrac{\alpha}{2}+\dfrac{1}{2},\dfrac{1}{2},x^2\right)+2\dfrac{\sqrt{d}h_0}{k}x M\left(\dfrac{\alpha}{2}+1,\dfrac{3}{2},x^2\right)\right]}. \label{f1Inf}
\end{align}
\endgroup

On the other hand, let us consider a new problem (P4) defined in the following way:  

\vspace{0.3cm}

\noindent \textbf{Problem (P4)}: Find the temperature $\Psi_{\infty}(x,t)$ and the moving interface $s_{\infty}(t)$ that satisfies:
\begingroup
\addtolength{\jot}{0.35em}
\begin{align}
& {\Psi_{\infty}}_t(x,t)=d {\Psi_{\infty}}_{xx}(x,t), \qquad 0<x<s_{\infty}(t), \quad t>0, \label{p4-1}\\
&  s_{\infty}(0)=0,\label{p4-2}\\
&  {\Psi_{\infty}}(s_{\infty}(t),t)=0, \qquad t>0, \label{p4-3}\\
&  {\Psi_{\infty}}(0,t)=T_{\infty}t^{\alpha/2}  \qquad t>0, \label{p4-4}\\
&  k{\Psi_{\infty}}_x(s_{\infty}(t),t)=-\gamma s_{\infty}(t)^{\alpha} \dot s_{\infty}(t), \qquad t>0, \label{p4-5}
\end{align}
\endgroup
As we can observe, this problem corresponds to a problem where a temperature boundary condition is imposed at the fixed face $x=0$. Thus the solution according to \cite{ZhXi} can be obtained from equations (\ref{49})-(\ref{52bis}):

\begingroup
\addtolength{\jot}{0.35em}
\begin{align}
&\Psi_{\infty}(x,t)=  t^{\alpha/2}\left[ {c_{11}}_{\infty} M\left(-\frac{\alpha}{2}, \frac{1}{2},-\eta^2\right)+{c_{21}}_{\infty} \eta M\left(-\frac{\alpha}{2}+\frac{1}{2},\frac{3}{2},-\eta^2\right)\right], \label{PsiInf}\\
& s_{\infty}(t)=2\nu_{\infty}  \sqrt{dt}, \label{sInf}
\end{align}
\endgroup
 where $\eta= \dfrac{x}{2\sqrt{dt}}$ and the constants ${c_{12}}_{\infty}$ and ${c_{22}}_{\infty}$ are given by: 

\begingroup
\addtolength{\jot}{0.35em}
\begin{align}
& {c_{11}}_{\infty}=T_{\infty}, \qquad \qquad \qquad  {c_{21}}_{\infty}=\dfrac{-T_{\infty}M\left( -\dfrac{\alpha}{2},\dfrac{1}{2},-{\nu_{\infty}}^2\right)}{\nu_{\infty} M\left(-\dfrac{\alpha}{2}+\dfrac{1}{2},\dfrac{3}{2},-{\nu_{\infty}}^2 \right)} \label{constInf}
\end{align}
\endgroup
and the parameter $\nu_{\infty}$ is the unique positive solution of the following equation:

\be \label{EcNuInfinito}
\dfrac{kT_{\infty}}{2^{\alpha+1} d^{\alpha/2+1} \gamma } f_2(x)=x^{\alpha +1},   \qquad x>0
\ee
with:
\be \label{f2}
f_2(x)=\dfrac{1}{x M\left(\dfrac{\alpha}{2}+1,\dfrac{3}{2},x^2 \right)} .
\ee

Once we have introduced the problems (P1) and (P4) we  are able to state the following convergence theorem.

\begin{teo}\label{TeoConvergence}
The problem (P1) converges to problem (P4) when $h_0$ tends to infinity, i.e:
\be
\lim\limits_{h_0\rightarrow +\infty} P1=P4 
\ee
In this context the term ``convergence" means that:

\begin{equation}
\left\lbrace 
\begin{array}{lll}
\lim\limits_{h_0\rightarrow +\infty} \nu_{h_0}&=& \nu_{\infty},\\
\lim\limits_{h_0\rightarrow +\infty} s_{h_0}(t)&=& s_{\infty}(t),\qquad \forall t>0 \\
\lim\limits_{h_0\rightarrow +\infty} \Psi_{h_0}(x,t)&=& \Psi_{\infty}(x,t),\quad \forall t>0, x>0.
\end{array}
\right.
\end{equation}

\end{teo}

\begin{proof}
Let us consider the problem (P1). We know that the parameter that characterizes the free boundary, $\nu_{h_0}$, is the unique solution of  equation (\ref{NuInf}). In order to obtain the limit of $\nu_{h_0}$ it is necessary to study the convergence of equation (\ref{NuInf}) when $h_0$ goes to infinity. The limit of the left hand side function of (\ref{NuInf}) is:

\begingroup
\addtolength{\jot}{0.35em}
\begin{align}
& \lim\limits_{h_0 \rightarrow +\infty} \left\lbrace \dfrac{h_0 T_{\infty}}{\gamma 2^{\alpha} d^{(\alpha+1)/2}}\dfrac{1}{\left[ \dfrac{1}{h_0}M\left(\dfrac{\alpha}{2}+\dfrac{1}{2},\dfrac{1}{2},x^2\right)+2\dfrac{\sqrt{d}}{k}x M\left(\dfrac{\alpha}{2}+1,\dfrac{3}{2},x^2\right)\right]h_0}	\right\rbrace \nonumber \\
& \qquad =  \dfrac{ T_{\infty}}{\gamma 2^{\alpha} d^{(\alpha+1)/2}}  \dfrac{1}{\left[2\dfrac{\sqrt{d}}{k}x M\left(\dfrac{\alpha}{2}+1,\dfrac{3}{2},x^2\right)\right]} \nonumber \\
& \qquad = \dfrac{k T_{\infty}}{2^{\alpha+1} d^{\alpha/2+1} \gamma} f_2(x) \label{EqNuInf}.
\end{align}
\endgroup
This imply that equation (\ref{NuInf}) converges to equation (\ref{EcNuInfinito}) when $h_0\rightarrow \infty$. On one hand, we have that the limit of $\nu_{h_0}$  must be a solution of equation (\ref{EcNuInfinito}). On the other hand, (\ref{EcNuInfinito}) has a unique solution $\nu_{\infty}$. Thus it turns out that $\lim\limits_{h_0\rightarrow\infty}\nu_{h_0}=\nu_{\infty}$. 
Once obtained this convergence, it is immediately that $\lim\limits_{h_0\rightarrow +\infty} s_{h_0}(t)= s_{\infty}(t)$, $\forall t>0$.  For the convergence of the temperature $\Psi_{h_0}(x,t)$ to $\Psi_{\infty}(x,t)$ when $h_0\rightarrow \infty$, it can be easily proved that: $\lim\limits_{h_0\rightarrow\infty}{c_{21}}(h_0)={c_{21}}_{\infty}$ and $\lim\limits_{h_0\rightarrow\infty}{c_{11}}(h_0)={c_{11}}_{\infty}$.

\end{proof}

\vspace{0.5cm}

\section{Numerical Computation}
\vspace{0.5cm}
From Theorem \ref{Teo2.1} the solution of the problem (P1) is characterized by a parameter $\nu$ defined as the unique solution of equation (\ref{18}). This equation can be rewritten into the following way:
\begin{equation}\label{78}
F(x)=\dfrac{h_0T_{\infty}}{\gamma 2^\alpha d^{(\alpha+1)/2}} f_1(x)-x^{\alpha+1}=0, \qquad x>0.
\end{equation}
where $f_1(x)$ is defined by (\ref{19}).

In order to approximate the unique root of the nonlinear equation defined above we can apply Newton's method. Beginning with an estimate $\nu_0$ of $\nu$, we define inductively: 

\begin{equation} \label{79}
\nu_{k+1}=\nu_k-\dfrac{F(\nu_k)}{F'(\nu_k)}
\end{equation}
where 

\begin{equation} \label{80}
F'(x)=\dfrac{h_0T_{\infty}}{\gamma 2^{\alpha} d^{(\alpha+1)/2}}f_1'(x)-(\alpha+1)x^{\alpha}.
\end{equation}
noting that:
\begin{equation} \label{81}
f_1'(x)=-f_1^2(x)\left[2(\alpha+1)x M\left(\dfrac{\alpha}{2}+\dfrac{3}{2},\dfrac{3}{2},x^2 \right) +2\dfrac{\sqrt{d}h_0}{k} M\left(\dfrac{\alpha}{2}+1,\dfrac{1}{2},x^2 \right)\right].
\end{equation}
We have  implemented Newton's Method using Matlab software. The main reason for choosing this programming language is that the Kummer function $M(a,b,z)$ can be represented by the command  `hypergeom'. The stopping criterion used is the boundedness of the  absolute error $\vert \nu_k-\nu_{k+1}\vert<10^{-15}$.
Without loss of generality we assume $\gamma=d=k=1$. The following Figures 1 to 4 present the computational values obtained for $\nu$ versus $h_0$ corresponding to different values of  $T_{\infty}$ and $\alpha$.

\newpage

\begin{multicols}{2}

\begin{Figure}
 \centering
 \includegraphics[width=1.1\textwidth]{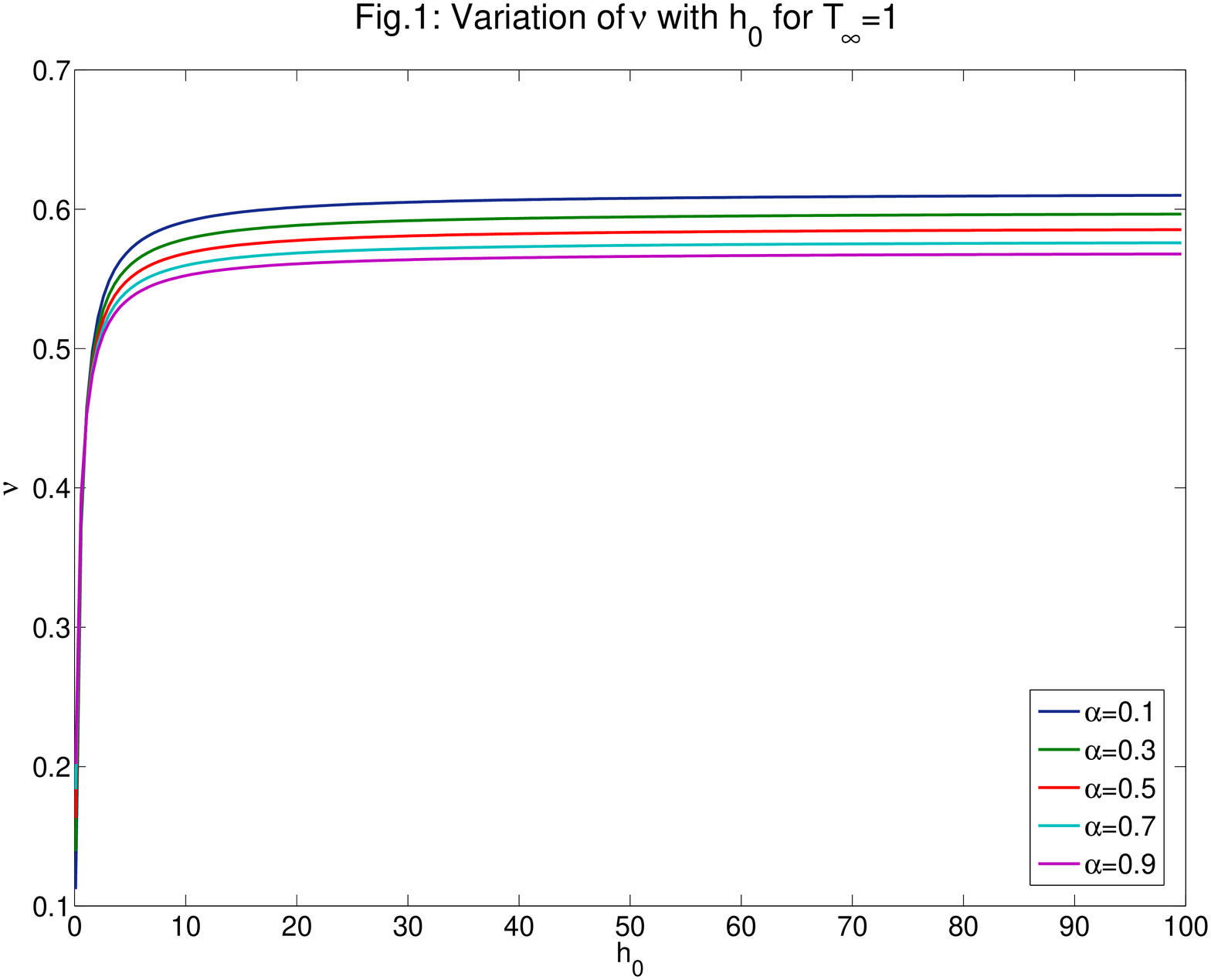}
\end{Figure}

\begin{Figure}
\includegraphics[width=1.1\textwidth]{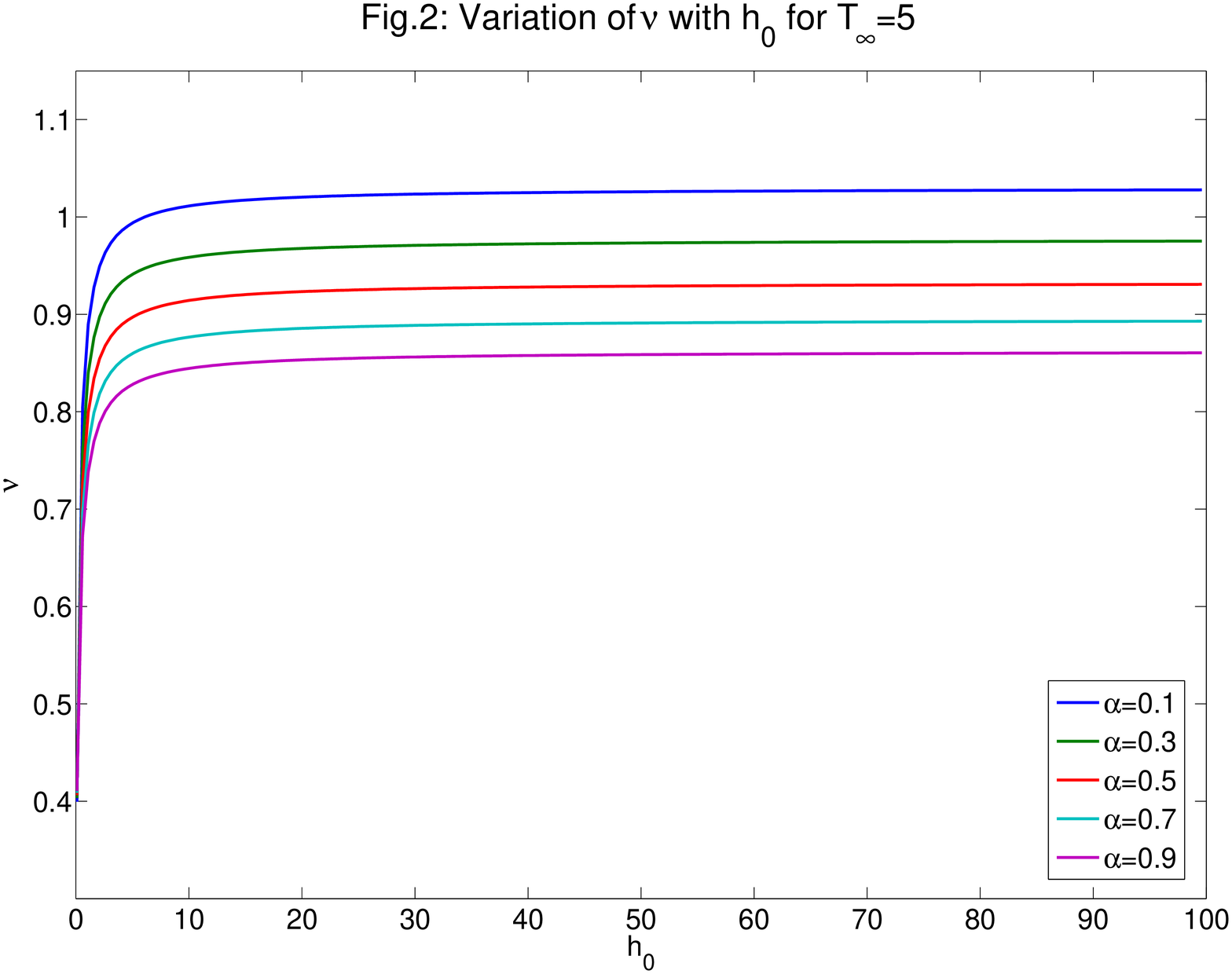}
\end{Figure}
\end{multicols}

\vspace{0.5cm}

\begin{multicols}{2}

\begin{Figure}
 \centering
 \includegraphics[width=1.1\linewidth]{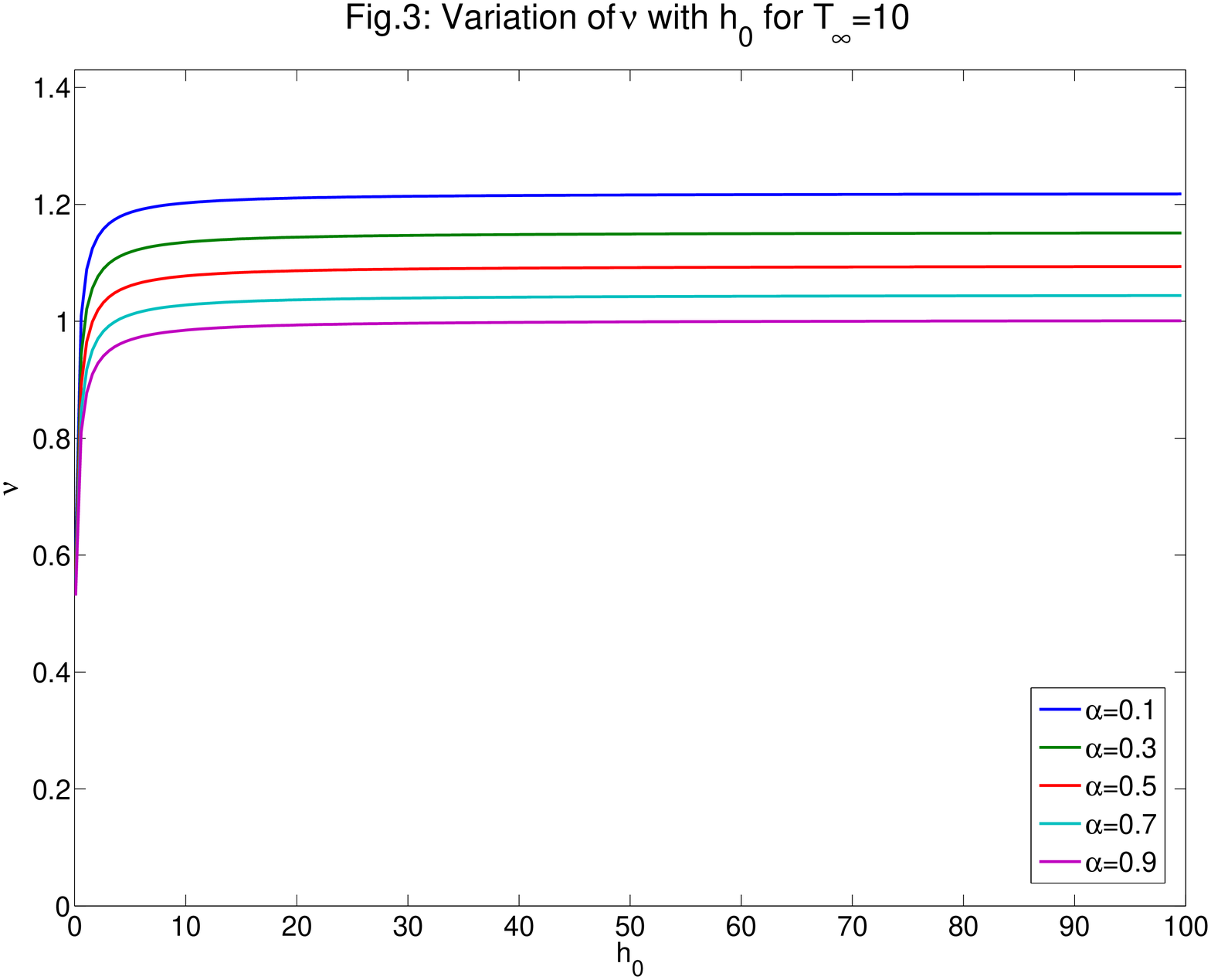}
\end{Figure}

\begin{Figure}
\includegraphics[width=1.1\linewidth]{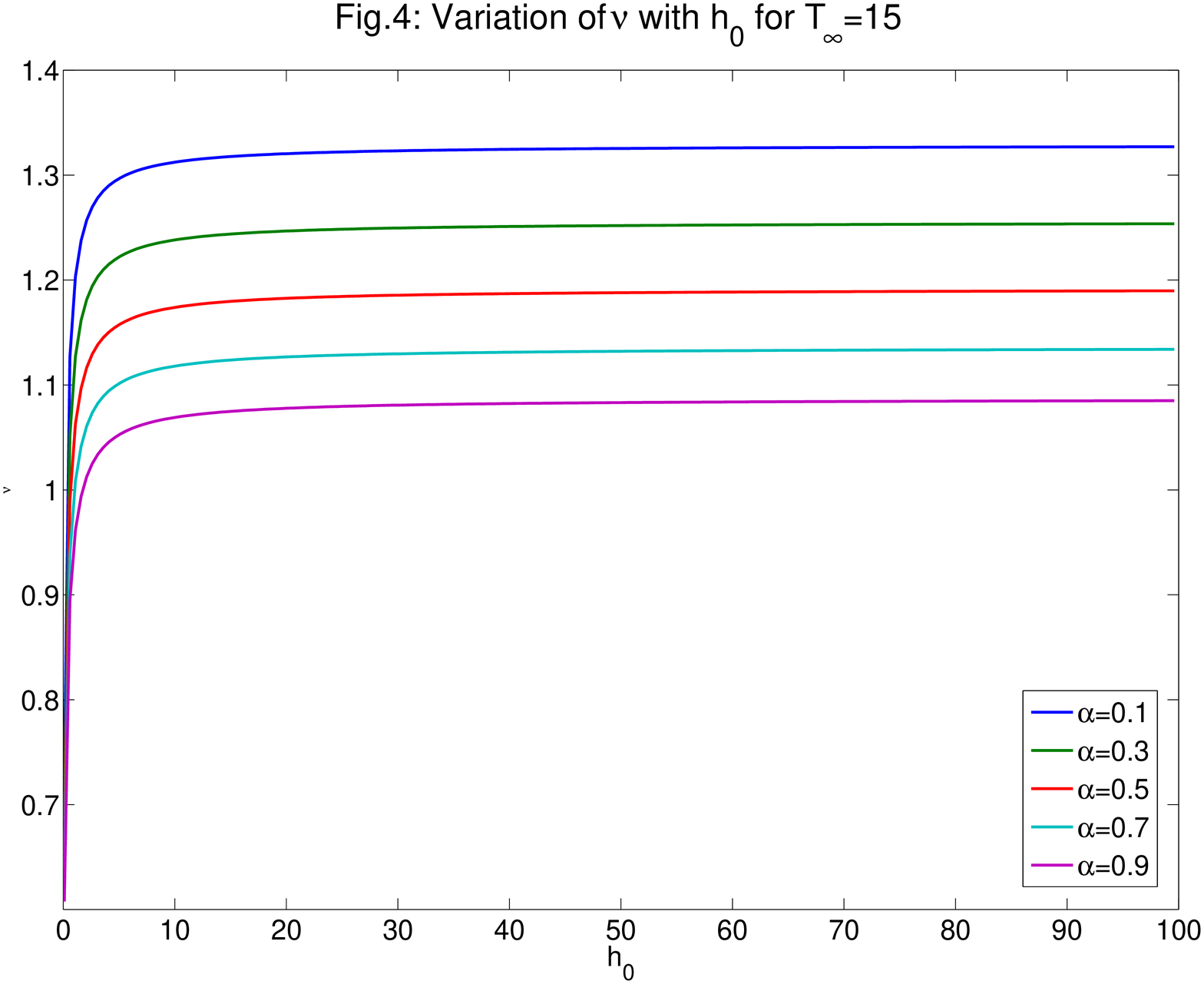}
\end{Figure}
\end{multicols}

\vspace{0.2cm}
\noindent We can observe that, in all cases $\nu$  varies monotonically increasing  with respect to $h_0$. In addition it can be appreciated that as $h_0$ increases, $\nu$ tends to stabilize. This behaviour is in accordance with Theorem \ref{TeoConvergence}, which ensures the existence of a limit for $\nu:=\nu_{h_0}$ when $h_0$ goes to infinity. For this reason, we also  applied Newton's method  to the problem (P4) taking into account equations (\ref{EcNuInfinito})-(\ref{f2}), using the same stopping criterion as above and taking $\gamma=d=k=1$.
In the next Figures 5 to 8, we compare the coefficients $\nu_{h_0}$ and $\nu_{\infty}$  corresponding to problems (P1) and (P4) respectively for different input data $T_{\infty}$ and $\alpha$. 

\vspace{0.5cm}

\begin{multicols}{2}

\begin{Figure}
\centering
\includegraphics[width=1.1\linewidth]{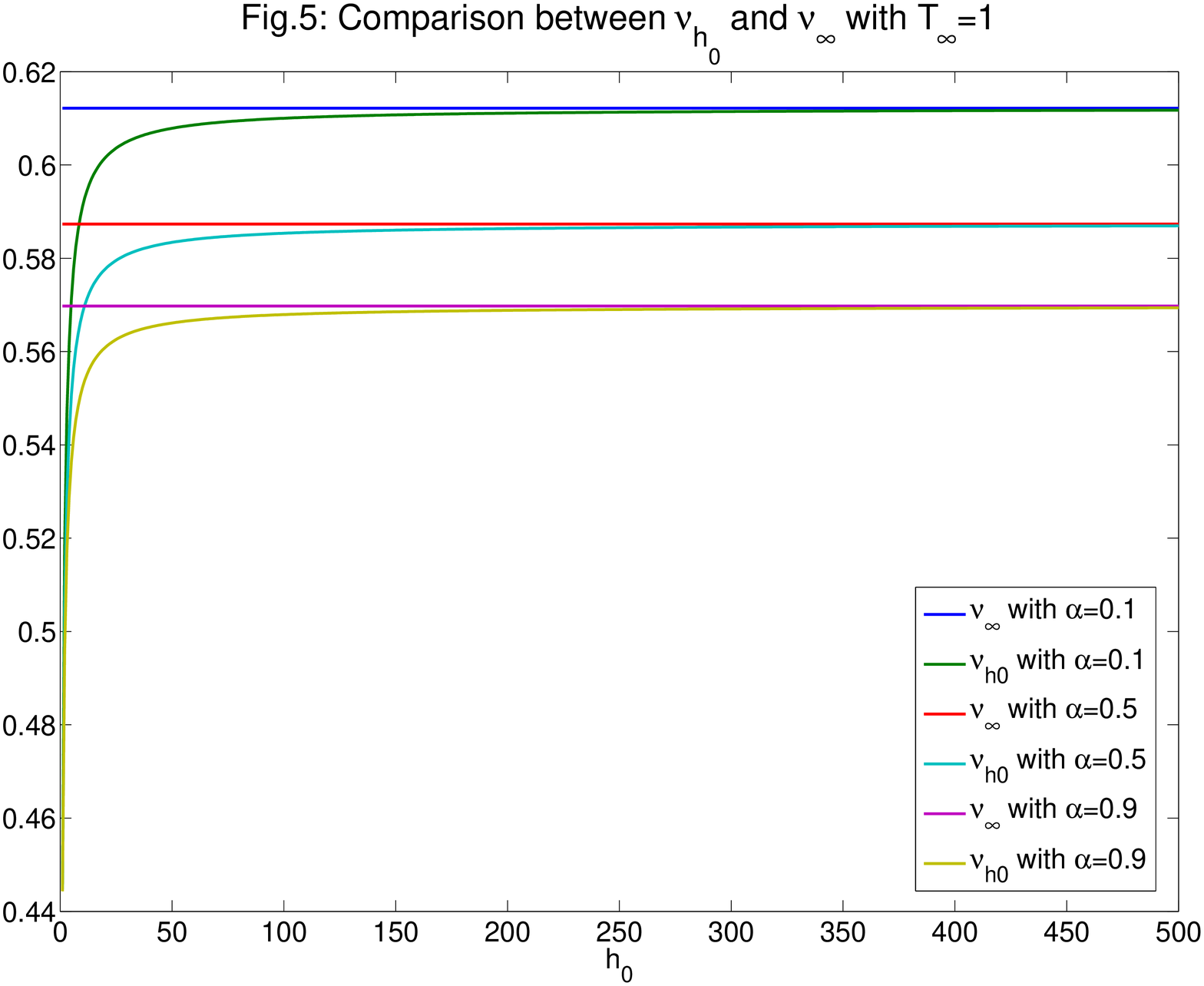}

\end{Figure}

\begin{Figure}
\centering
\includegraphics[width=1.1\linewidth]{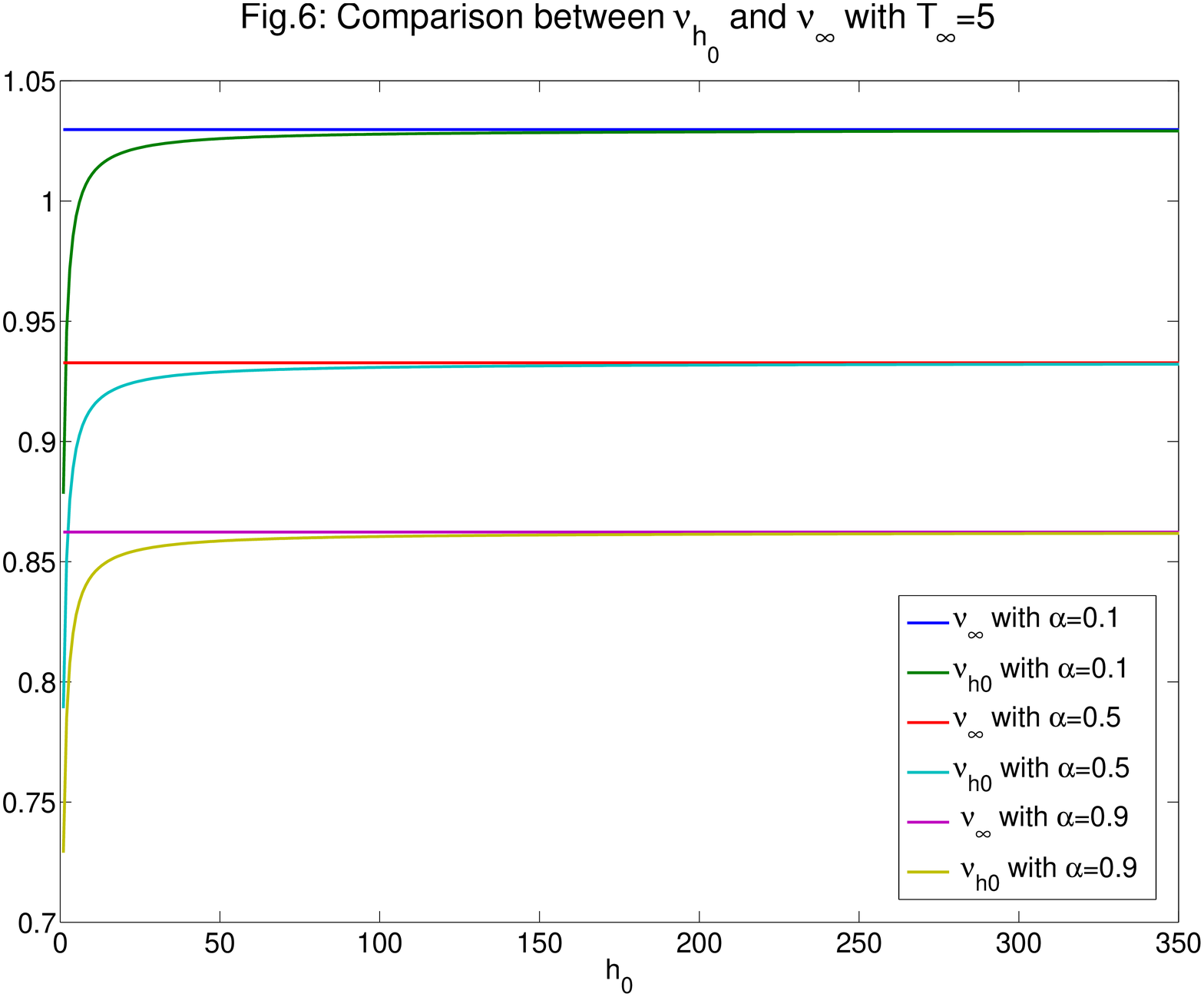}

\end{Figure}
\end{multicols}

\vspace{0.5cm}

\begin{multicols}{2} 
\begin{Figure}
\centering
\includegraphics[width=1.1\linewidth]{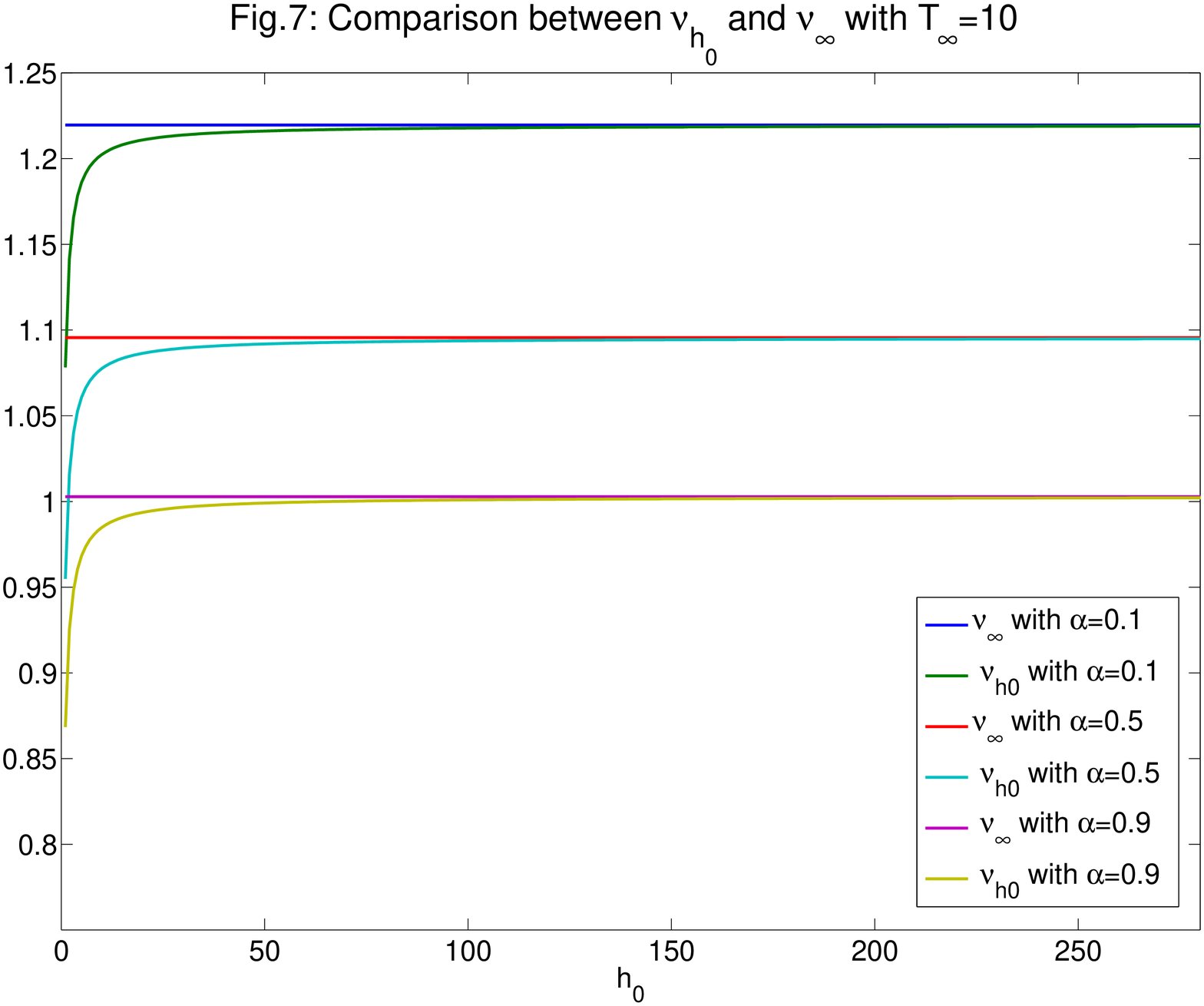}

\end{Figure}

\begin{Figure}
\centering
\includegraphics[width=1.1\linewidth]{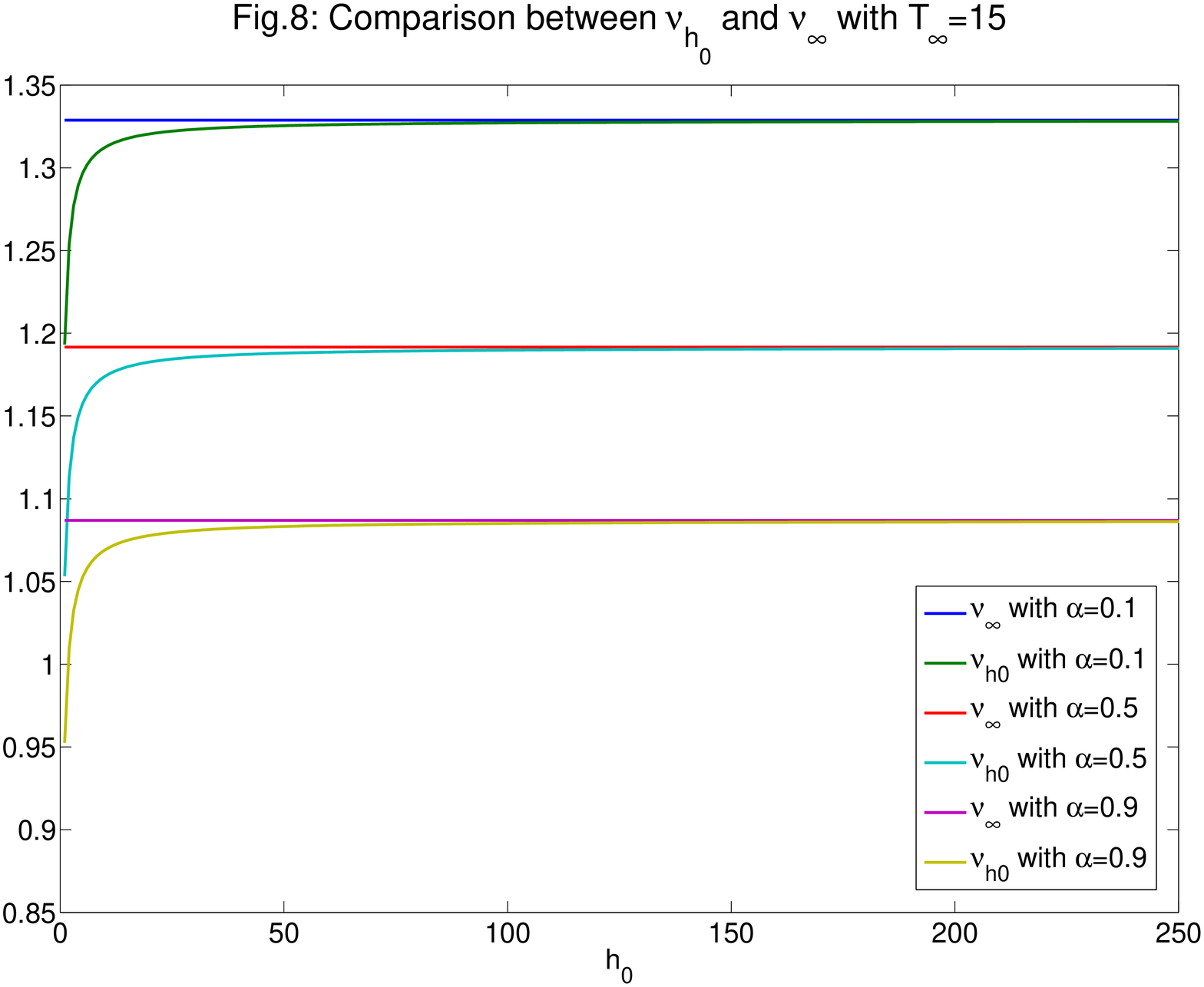}

\end{Figure}

\end{multicols}
\vspace{0.2cm}

In Figure 9 we show the variation of the temperature $\Psi$ with respect to $x$ and $t$ taking the particular values of the data: $\gamma=k=d=1$, $\alpha=0.4$, $h_0=0.5$ and $T_{\infty}=1$. As we are dealing with a melting problem, for every fixed value of the position (x) we can note when the phase-change takes place and  observe how the temperature becomes greater over time  once the phase-change have occurred. 
\begin{figure}[h]
\centering
\includegraphics[scale=0.53]{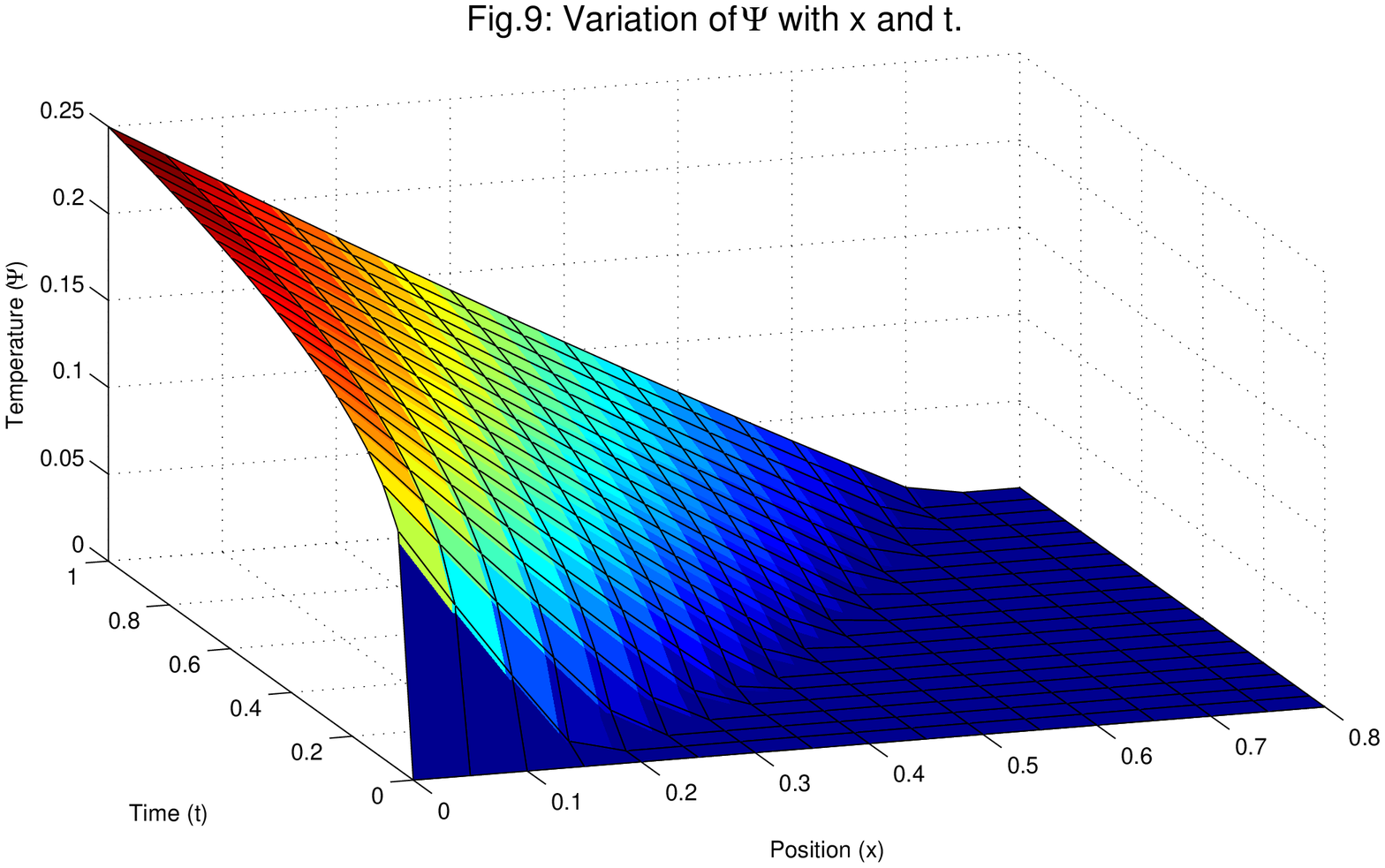}
\end{figure}

\vspace{0.5cm}

\section{Conclusions} $ $

 In this article a closed analytical solution of a similarity type have been obtained for a one-dimensional one-phase Stefan problem in a semi-infinite material  using Kummer functions. The novel feature in the problem studied concerns a variable latent heat that depends on the position as well as a convective boundary condition at the fixed face $x=0$ of the material. On one hand, assuming a latent heat defined as a power function of the position allows the generalization of some previous theoretical results, finding its physical base in problems related to the movement of a shoreline or  the cooling body of a magma. On the other hand, the fact of considering a convective condition at the fixed boundary reflects a more realistic way of heat input than an imposed temperature or flux, known as Dirichlet and Neumann conditions respectively. 

The key contribution of this paper has been to present the exact solution of the problem which is worth finding not only to understand better the process involved but also to verify the consistency and estimate errors of numerical methods designed to solve Stefan problems.  We have demonstrated also the equivalence between our problem and the problems defined by considering a temperature or a flux boundary condition instead of the convective one.  

Besides, it has been analysed the limit behaviour of the solution when the coefficient $h_0$ that characterizes the heat transfer at the fixed face $x=0$ tends to infinity. It can be said that our problem (P1) converges  pointwise to a problem (P4) where it is prescribed a temperature at the fixed boundary characterized by   $T_{\infty}$.

Finally, we have applied Newton's Method to the closed formula obtained for our problem (P1), in order to   estimate the parameter that characterizes the free front numerically. In the same way we did to problem (P4).  The computations obtained help us to validate our convergence result.

\vspace{0.5cm}

\section*{Aknowledgements}

The present work has been partially sponsored by the Project PIP No 0534 from CONICET-UA, Rosario, Argentina, and Grant AFORS FA9550-14-1-0122.

%%=========================================
\small{

}

\end{document}